\documentclass[11pt]{amsart}
\usepackage{amssymb}
\usepackage{graphicx}
\newtheorem{theorem}{Theorem}

\newtheorem{lemma}{Lemma}

\theoremstyle{definition}
\newtheorem{remark}{Remark}

\begin{document}
\baselineskip17pt

\title{Dynamics of antibody levels: asymptotic properties}
\author[K. Pich\'or]{Katarzyna Pich\'or}
\address{K. Pich\'or, Institute of Mathematics,
University of Silesia, Bankowa 14, 40-007 Kato\-wi\-ce, Poland.}
\email{katarzyna.pichor@us.edu.pl}
\author[R. Rudnicki]{Ryszard Rudnicki}
\address{R. Rudnicki, Institute of Mathematics,
Polish Academy of Sciences, Bankowa 14, 40-007 Katowice, Poland.}
\email{rudnicki@us.edu.pl}

\begin{abstract}
We study properties of a piecewise deterministic Markov process modeling 
the changes in  concentration of specific antibodies.
The evolution of densities of the process is described by a stochastic semigroup.
The long-time behaviour of this semigroup is studied. In particular
we prove theorems on its asymptotic stability.
\end{abstract}

\let\thefootnote\relax\footnote{This research was partially supported by
the  National Science Centre (Poland) Grant No. 2017/27/B/ST1/00100.}

\maketitle

\section{Introduction}
\label{intro}
In \cite{DGKT} the authors introduced a mathematical model of the immune system.
The immune status is the concentration of specific antibodies, which appear after infection with a pathogen
and remain in serum, providing protection against future attacks of that same pathogen. Over time 
the number of  antibodies decreases until the next infection.  
During fighting the invader the immunity is boosted and then 
the immunity is gradually waning, etc. 
Thus the concentration of antibodies is described by a stochastic process whose 
trajectories are decreasing functions $x(t)$ between subsequent infections.
These functions satisfy the differential equation
\begin{equation}
\label{wane}
x'(t)=g(x(t)).
\end{equation}
It is assumed that the time it takes the immune system to clear infection is negligible
and that if $x$ is the concentration of antibodies at the moment of infection, then $Q(x)>x$ is the concentration of antibodies 
just after clearance of infection. 
An explicit expression for $Q$ was derived in \cite{GKTD,TEGB-MK}.
It is also assumed that the moments of infections are independent of the state of the immune system and they
are distributed according to a Poisson process $(N_t)_{t\ge 0}$
 with rate $\Lambda>0$. 

The immune status is a flow on the interval $[0,\infty)$ with jumps at random moments $t_0<t_1<t_2<\dots$ (see Fig.~\ref{r:imun-syst1}).
Such a flow belongs to the family of piecewise deterministic Markov processes 
\cite{davis84,RT-K-k}. We denote this process by $(\xi_t)_{t\ge 0}$ and it is defined by the following equations
\[
\xi_{t_n}=Q(\xi_{t_{n-1}^-}),\quad \xi_t'=g(\xi_t)\,\,\,\textrm{for $t\in [t_{n-1},t_n)$},\quad N_{t_n}=N_{t_n^-}+1=n.
\]
It means that the process $(\xi_t)_{t\ge 0}$ satisfies
the following stochastic differential equation
\[
d\xi_t=g(\xi_t)\,dt+(Q(\xi_t)-\xi_t)\,dN_t.
\]

\begin{figure}
\begin{center}
\begin{picture}(340,105)(20,20)
\includegraphics[viewport=115 554 492 703]{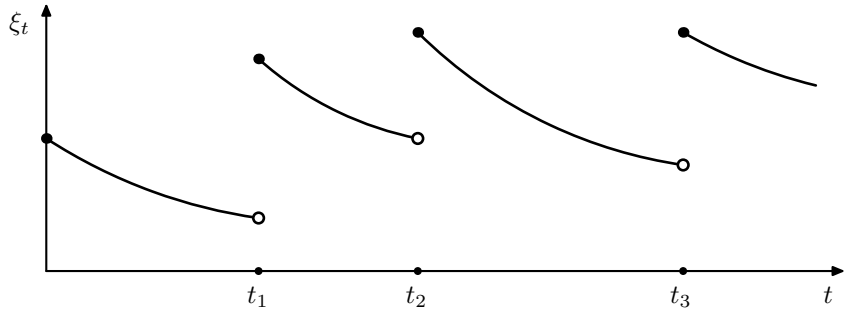}
\end{picture}
\end{center}
\caption{An example of a trajectory of the process $\xi_t$.}
\label{r:imun-syst1}
\end{figure}

One of the main interesting problems is the evolution of the distribution of this process, 
in particular the existence of a unique stationary density $f^*$ and its asym\-ptotic stability.
It is worth to mention that if the process $(\xi_t)_{t\ge 0}$ has a unique stationary density $f^*$
then, according to the ergodic theorem, $f^*$ is the density of  distribution of the immune status in the population.  
In \cite{DGKT} the asymptotic stability of a stationary density $f^*$ was proved
for a function $Q$ which is unimodal and has properties: $\lim_{x\to0}Q(x)=\infty$
and $\lim_{x\to\infty}(Q(x)-x)=\text{const}$.

The aim of this note is to show that asymptotic stability holds for a large class of $C^1$-functions $Q$.
In particular we extend the result from the paper \cite{DGKT} to the significant case when the increase of the concentration of antibodies 
after the infection is bounded. Moreover, we present another technique to prove this result,
which seems to be easier in applications because it does  not require to prove directly the existence of an invariant density.
The main idea of the paper is to formulate the problem in the terms of stochastic semigroups and then apply some results concerning
the Foguel alternative~\cite{PR-JMMA2016,PR-SD2017}, which gives conditions when a stochastic semigroup is asymptotically stable or sweeping.

The organization of the paper is as follows.
In section~\ref{s:model} we present the assumptions concerning our model and formulate the main problem in the terms of stochastic semigroups.
Section~\ref{s:asympt} contains the definitions and results concerning
asymptotic properties of stochastic semigroups and the proof of the main result of the paper.
In the last section we discuss the  case when concentration of antibodies is bounded and we give some examples.

\section{A semigroup formulation of the problem}
\label{s:model}
Concerning $g$ and $Q$ we assume the following
\begin{itemize}
\item[(A1)] $g\colon [0,\infty)\to \mathbb R$ is a $C^1$-function such that $g(x)<0$ for $x>0$ and $g(0)=0$,
\item[(A2)] $Q\colon [0,\infty)\to (0,\infty)$ is a $C^1$-function such that $Q(x)>x$ for $x\ge 0$, 
\item[(A3)] $|A|=0 \Longrightarrow  |Q^{-1}(A)|=0$, where $A$ is a Borel subset of $[0,\infty)$ and $|\cdot|$ denotes the Lebesgue measure.
\end{itemize}
We denote by $\pi_tx_0$ the solution $x(t)$ of Eq.~\eqref{wane} with the initial condition $x(0)=x_0\ge 0$.

Assumption (A3) allows us to introduce   \cite{LiM,Rudnicki-LN} a linear operator $P_Q$ on the space $L^1=L^1[0,\infty)$ given by the formula
\begin{equation}
\label{def-FP}
\int_A P_Q f(x)\,dx=
\int_{Q^{-1}(A)} f(x)\,dx
\end{equation}
for each $f\in L^1$ and all Borel subsets $A$ of $[0,\infty)$.
The operator $P_Q$
is called the
{\it Frobenius--Perron operator} for the transformation~$Q$. 
The adjoint of the Frobenius--Perron operator
$P^*\colon L^{\infty}[0,\infty)\to L^{\infty}[0,\infty)$ is 
given by $P^*g(x)= g(Q(x))$
and it is called the \textit{Koopman operator} or the 
\textit{composition operator}.

Denote by $D$ the subset of the space
$L^1$ which contains all
\textit{densities}
\[
D=\{f\in L^1\colon \,\, f\ge 0,\,\, \|f\|=1\}.
\]
The Frobenius--Perron operator describes the evolution of densities under the action of the transformation $Q$ and it is an example of a
\textit{stochastic} or \textit{Markov  operator},
which is defined as a linear operator $P\colon  L^1\to L^1$ 
such that $P(D)\subset D$. 

The class of the functions $Q$ which satisfy (A3) is rather large.  For example if $Q$ is a $C^1$-function and there exists an at most countable family of intervals $[a_i,b_i]$, $i\in I$, such that 
\[
[0,\infty)=\bigcup_{i\in I}[a_i,b_i],\quad (a_i,b_i)\cap (a_j,b_j)=\emptyset\quad \text{for $i\ne j$}
\]
and $Q'(x)\ne 0$ for $x\in (a_i,b_i)$ and $i\in I$,
then $Q$ satisfies (A3) and the operator $P_Q$ is given by the formula
\begin{equation}
\label{F-P-operator-gladki}
P_Qf(x)=\sum_{i\in I_x} f(\varphi_i(x))|\varphi_i'(x)|,
\end{equation} 
where $\varphi_i$ is the inverse function of $Q\big|_{(a_i,b_i)}$
and $I_x=\{i\colon \varphi_i(x)\in (a_i,b_i)\}$.

Now we add the second ingredient to the model. 
If $f$ is the initial density of immune status and there is no infection till the time $t$, then 
the density of immune status at $t$ is given by $S(t)f$, where $S(t)$ is the Frobenius-Perron operator related to 
the transformation $x\mapsto \pi_tx$. In this way we obtain a $C_0$-semigroup of stochastic operators $\{S(t)\}_{t\ge 0}$ given by
\[
S(t)f(x)=
\begin{cases}
f(\pi_{-t}x)\dfrac{\partial (\pi_{-t}x)}{\partial x}&\text{if $\pi_{-t}x$ exists,}\\
0&\text{if $\pi_{-t}x$ does not exist}.
\end{cases}
\] 
The semigroup  $\{S(t)\}_{t\ge 0}$ has the infinitesimal generator $\mathcal A_0f(x)=-\dfrac{d}{dx}(g(x)f(x))$
with the domain $\mathfrak D(\mathcal A_0)=\{f\in L^1\colon \mathcal A_0f\in  L^1\}$. 
Here the notation $\mathcal A_0f\in  L^1$ means that $f$ is a locally absolutely continuous function, so $f'$ exists a.e., and $(gf)'\in L^1$.
The adjoint semigroup  $\{S^*(t)\}_{t\ge 0}$ on $L^{\infty}$ is given by the formula $S^*(t)f(x)=f(\pi_tx)$.

Finally, we combine both ingredients: waning and boosting of immunity status. Then the density $u(t)=u(t,x)$ of immune status  
satisfies the following evolution equation in $L^1$
\begin{equation}
\label{evol-eq}
u'(t)=\mathcal Au(t),
\end{equation}
where $\mathcal A=\mathcal A_0+\Lambda P_Q-\Lambda I$.
The solution $u(t)$ of this equation generates a \textit{stochastic semigroup} (i.e. a $C_0$-semigroup of stochastic operators) 
$\{U(t)\}_{t\ge 0}$. It means that if $f$ is the density of initial immune status then $U(t)f$ is the density of immune status at time $t$.
The semigroup $\{U(t)\}_{t\ge 0}$ is given by the \textit{Dyson-Phillips expansion}
\begin{equation}
\label{eq:dpf1b}
U(t)f=e^{-\Lambda t}\sum_{n=0}^\infty \Lambda^n S_n(t)f, 
\end{equation}
where
\begin{equation}
\label{eq:dpf2b}
S_0(t)f=S(t)f,\quad S_{n+1}(t)f=\int_0^tS_{n}(t-\tau)P_QS_0(\tau)f\,d\tau, \quad
n\ge 0.
\end{equation}
Similar formulas to \eqref{eq:dpf1b}--\eqref{eq:dpf2b} hold for the adjoint semigroup $\{U^*(t)\}_{t\ge 0}$ on $L^{\infty}$.
In particular if $f\ge 0$ then
\begin{equation}
\label{eq:oszac}
U^*(t)f\ge \Lambda e^{-\Lambda t}\int_0^t S^*(t-\tau)P^*_QS^*(\tau)f\,d\tau=\Lambda e^{-\Lambda t}\int_0^t  f(\pi_{\tau}Q(\pi_{t-\tau}x))\,d\tau.
\end{equation}
The  process $(\xi_t)_{t\ge 0}$  has the probability transition function $\mathcal P(t,x,\Gamma)$
given by 
\[
\mathcal P(t,x,\Gamma)=U^*(t)\mathbf 1_{\Gamma}(x) .
\]
Now inequality \eqref{eq:oszac} allows us to estimate $\mathcal P(t,x,\Gamma)$ from below 
\begin{equation}
\label{eq:oszac-tr-prob}
\mathcal P(t,x,\Gamma)\ge \Lambda e^{-\Lambda t}\int_0^t \mathbf 1_{\Gamma}(\pi_{\tau}Q(\pi_{t-\tau}x))\,d\tau.
\end{equation}
This inequality will play the crucial role in the proof of the existence and asymptotic stability of a stationary density.

\section{Asymptotic stability and sweeping}
\label{s:asympt} 
We start with some general definitions and results concerning asymptotic stability and sweeping of stochastic semigroups.

Let a triple $(X,\Sigma,\mu)$ be a $\sigma$-finite measure space.
A stochastic semigroup $\{P(t)\}_{t\ge 0}$ on $L^1=L^1(X,\Sigma,\mu)$ is called
{\it asymptotically stable} if  there exists a density $f^*$   such that
\begin{equation}
\label{d:as}
\lim _{t\to\infty}\|P(t)f-f^*\|=0 \quad \text{for}\quad f\in D.
\end{equation}
If the semigroup $\{P(t)\}_{t\ge 0}$ is generated by some
evolution equation $u'(t)=Au(t)$ then  
the asymptotic stability of $\{P(t)\}_{t\ge 0}$ 
means that the stationary solution $u(t)=f^*$ is asymptotically stable
in the sense of Lyapunov and this stability is global on the set $D$.

A stochastic semigroup $\{P(t)\}_{t\ge 0}$
is called {\it partially integral} if
there exist $t>0$ and a  measurable function $q(t,\cdot,\cdot)\colon X\times X\to [0,\infty)$ such that 
\begin{equation}
\label{w-eta0}
P(t)f(y)\ge \int_X q(t,x,y)f(x)\, \mu(dy)\quad \textrm{for $f\in D$},
\end{equation}
and
\begin{equation*}
\int_X\int_X  q(t,x,y)\,\mu(dx)\,\mu(dy)>0.
\end{equation*}
If $\mathcal P(t,x,dy)$ is the transition probability function corresponding to the stochastic semigroup $\{P(t)\}_{t\ge 0}$ 
then inequality \eqref{w-eta0} can be written  in an equivalent form $\mathcal P(t,x,dy)\ge q(t,x,y)\,dy$.
We will use the following criterion of asymptotic stability.

\begin{theorem}\cite{PR-jmaa2}
\label{asym-th2}
Let $\{P(t)\}_{t\ge 0}$ be a partially integral stochastic
semigroup. Assume that the  semigroup $\{P(t)\}_{t\ge 0}$ has
a unique invariant density $f^*$. If $f^*>0$ a.e., then the semigroup
$\{P(t)\}_{t\ge 0}$ is asymptotically stable.
\end{theorem}

A stochastic semigroup $\{P(t)\}_{t\ge 0}$ is
called \textit{sweeping}
with respect to a set $B\in\Sigma$ if for every
 $f\in D$
\begin{equation*}
\lim_{t\to\infty}\int_B P(t)f(x)\,\mu(dx)=0.
\end{equation*}
From now on we assume additionally  that $(X,\rho)$ is a separable metric space and   $\Sigma=\mathcal B(X)$ is the $\sigma$-algebra of Borel subsets of $X$.
We will consider stochastic semigroups $\{P(t)\}_{t\ge 0}$ 
which satisfy the following condition:

\noindent (K) for every $x_0\in X$ there exist  an $\varepsilon >0$,  a $t>0$,
and a measurable function 
$\eta\ge 0$ such that $\int \eta(x)\, \mu(dx)>0$ and
\begin{equation}
\label{w-eta3}
\mathcal P(t,x,dy)\ge \eta(y)\,\mu(dy) \quad \textrm{for $x\in B(x_0,\varepsilon)$},
\end{equation}
where
$B(x_0,\varepsilon)=\{x\in X:\,\,\rho(x,x_0)<\varepsilon\}$.

It is clear that if a stochastic semigroup satisfies condition (K) then it is partially integral.
We will need the following criterion of sweeping \cite[Corollary 2]{PR-JMMA2016}.
\begin{theorem}
\label{col-sw}
Assume that a stochastic semigroup $\{P(t)\}_{t\ge 0}$ satisfies
condition {\rm (K)} and has no invariant densities.
Then $\{P(t)\}_{t\ge 0}$ is sweeping from compact sets.
\end{theorem}

We say that a stochastic semigroup $\{P(t)\}_{t\ge 0}$ satisfies the {\it Foguel alternative} if it
is asymptotically stable or sweeping from all compact sets  \cite{LiM}.
We now formulate the main result of this paper. 
\begin{theorem}
\label{Foguel-al}
The semigroup $\{U(t)\}_{t\ge 0}$ satisfies the Foguel alternative, i.e.
it is asymptotically stable or for every $f\in L^1[0,\infty)$ and $M>0$
\[
\lim_{t\to\infty}\int_0^M U(t)f(x)\,dx=0.
\]
\end{theorem}

In order to prove Theorem~\ref{col-sw} it is enough to check that the semigroup $\{U(t)\}_{t\ge 0}$ satisfies condition (K) and 
that if $f^*$ is an invariant density for $\{U(t)\}_{t\ge 0}$ then  $f^*(x)>0$ a.e. 
Indeed, if  $\{U(t)\}_{t\ge 0}$  has no invariant densities, then according to Theorem~\ref{col-sw} this 
semigroup is sweeping from compact sets. In the case when  $\{U(t)\}_{t\ge 0}$  has more then one invariant density
then it is easy to construct two invariant densities  $f_1^*$ and $f_2^*$ with disjoint supports, i.e. such that  $f_1^*f_2^*=0$ a.e.
Thus, the uniqueness of an invariant density will be a simple consequence of its strict positivity.
It means that if an invariant density  exists and we know that this density has to be positive then according to Theorem~\ref{asym-th2}  
the semigroup $\{U(t)\}_{t\ge 0}$ is asymptotically stable.

\begin{lemma}
\label{L:cond-K}
The semigroup $\{U(t)\}_{t\ge 0}$ fulfills condition \rm{(K)}. 
\end{lemma}

\begin{proof}
From \eqref{eq:oszac-tr-prob} it follows that
\begin{equation}
\label{eq:oszac-tr-prob2}
\mathcal P(t,x,\Gamma)
\ge \Lambda e^{-\Lambda t}\int_0^t \mathbf 1_{\Gamma}(r(\tau,t,x))\,d\tau,
\end{equation}
where  $r(\tau,t,x)=\pi_{\tau}Q(\pi_{t-\tau}x)$.
First we want to find the derivative $\dfrac{\partial r}{\partial \tau}$.
We use the following formulas:
\begin{equation}
\label{eq:poch-pi}
\frac{\partial }{\partial t}(\pi_tx)=g(\pi_tx),\quad
\frac{\partial }{\partial x}(\pi_tx)=\frac{g(\pi_tx)}{g(x)}.
\end{equation}
The first formula follows directly from the definition of $\pi_t x$. Now we derive the second one. Let 
$\varphi(t,x)=\dfrac{\partial }{\partial x}(\pi_tx)$. Then
\[
\frac{\partial \varphi}{\partial t}(t,x)= \dfrac{\partial }{\partial x}\frac{\partial }{\partial t}(\pi_tx)= 
\dfrac{\partial }{\partial x}g(\pi_tx)=g'(\pi_tx)\frac{\partial }{\partial x}(\pi_tx)=g'(\pi_tx)\varphi(t,x).
\]
Hence  
\[
\frac{\partial }{\partial t}\left(\ln\varphi(t,x)\right)=g'(\pi_tx),
\]
but since $\varphi(0,x)=1$ we have
\[
\ln\varphi(t,x)=\int_0^t g'(\pi_sx)\,ds =\int_x^{\pi_tx} \frac{g'(y)}{g(y)}\,dy =\ln g(y)\bigg|_{y=x}^{y=\pi_tx}=\ln\left(\frac{g(\pi_tx)}{g(x)}\right),
\]
which proves the second formula of (\ref{eq:poch-pi}).
From the chain role we obtain
\[
\dfrac{\partial r}{\partial \tau}(\tau,t,x)
=g(\pi_{\tau}Q(\pi_{t-\tau}x))- 
\frac{g(\pi_{\tau}Q(\pi_{t-\tau}x))}{g(Q(\pi_{t-\tau}x))}Q'(\pi_{t-\tau}x)g(\pi_{t-\tau}x).
\]
If $\tau=0$ and $x=x_0$, then 
\[
\lim_{t\to\infty}\dfrac{\partial r}{\partial \tau}(0,t,x_0)
=\lim_{t\to\infty}[g(Q(\pi_tx_0))- 
Q'(\pi_tx_0)g(\pi_tx_0)]=g(Q(0)).
\]
Since $g(Q(0))< 0$ and $r$ is a $C^1$ function we can find a sufficiently large $t$
and positive constants $\varepsilon'$, $\delta$, $M$, and $\tau_0\le t$  
such that 
\[
-M \le \dfrac{\partial r}{\partial \tau}(\tau,t,x)  \le -\delta\quad\textrm{for $\tau\in [0,\tau_0]$ and $x\in B(x_0,\varepsilon')$}.
\]
From \eqref{eq:oszac-tr-prob2} it follows that
\begin{equation}
\label{eq:oszac-tr-prob3}
\mathcal P(t,x,\Gamma)
\ge \Lambda e^{-\Lambda t}\int_0^{\tau_0} \mathbf 1_{\Gamma}(r(\tau,t,x))\,d\tau
\ge \frac{\Lambda e^{-\Lambda t}}M\int\limits_{\Delta_x}\mathbf 1_{\Gamma} (y)\,dy
\end{equation}
for $x\in B(x_0,\varepsilon')$, where $\Delta_x =[r(\tau_0,t,x),r(0,t,x)]$.
The interval $\Delta_x$ has  the length at least $\delta\tau_0$.  
Let $\varepsilon\in (0,\varepsilon')$ be such that 
\[
|r(0,t,x)-r(0,t,x_0)|<\delta\tau_0/3\quad \text{for $x\in B(x_0,\varepsilon)$}. 
\]
Then we find an interval $\Delta$ with a length of at least  $\delta\tau_0/3$
such that  $\Delta \subset \Delta_x$ for $x\in B(x_0,\varepsilon)$.
Let $\eta(y)=\Lambda e^{-\Lambda t}M^{-1}\mathbf 1_{\Delta}(y)$.
Then $\mathcal P(t,x,dy)\ge \eta(y)\,dy$ for $x\in B(x_0,\varepsilon)$.
\end{proof}

\begin{lemma}
\label{L:positivity-of-density}
If $f^*$ is an invariant density with respect to $\{U(t)\}_{t\ge 0}$, then $f^*>0$ a.e.
\end{lemma}

\begin{proof}
Let $A=\{x\colon f^*(x)>0\}$. The set $A$ is defined up to a set of measure zero.
Since
\[
f^*(x)=U(t)f^*(x) \ge e^{-\Lambda t} S(t)f^*(x)= e^{-\Lambda t}f^*(\pi_{-t}x)\dfrac{\partial \pi_{-t}x}{\partial x}>0
\]
for $x\in \pi_t(A)$ and $t\ge 0$, we have $\pi_t(A) \subseteq A$ for arbitrary $t>0$, and consequently  $A=(0,a)$ or $A=(0,\infty)$.
We check that  $A=(0,\infty)$. Assume on the contrary that $A=(0,a)$.
Then $S(\tau)f^*(x)>0$ for  $x\in (0,b)$, $b=\pi_{\tau}(a)$.
Let $m=\min\{Q(x)\colon x\ge 0\}$ and assume that $Q(a)\ne m$. 
Observe that if  $f(x)>0$ for $x\in (0,b)$, 
then $P_Qf(x)>0$ for all $x\in (m,Q(b))$. 
It means that  
\[
S(t-\tau)P_QS(\tau)f^*(x)>0 \quad\text{for $x\in (\pi_{t-\tau}m, \pi_{t-\tau}Q(\pi_{\tau}(a)))$}.
\]
Since 
\[
f^*(x) \ge \Lambda e^{-\Lambda t} S_1(t)f^*(x)= \Lambda e^{-\Lambda t}\int_0^t S(t-\tau)P_QS(\tau)f^*(x)\,d\tau,
\]
we have $f^*(x)>0$ for $x\in (\pi_tm,\pi_tQ(a))$.
As $m<Q(a)$, the interval $I_t=(\pi_tm,\pi_tQ(a))$ is nontrivial.  
Moreover, $\pi_tQ(a)>a$ for sufficiently small $t>0$, 
which contradicts the definition of $A$.
In the case $Q(a)=m$ we need an extra argument. 
From assumption (A3) it follows that the transformation $Q$ cannot be  
constant on any nontrivial interval. Let $\overline m=\max\{Q(x)\colon x\le a\}$.
If $f(x)=P_QS(\tau)f^*(x)$, then $f(x)>0$ for $x\in Q((0,\pi_{\tau}a))$.
We can find an $\varepsilon>0$ such that $[\overline m-\varepsilon,\overline m]\subset Q((0,\pi_{\tau}a))$
for sufficiently small $\tau>0$. Hence    
$S(t-\tau)P_QS(\tau)f^*(x)>0$ for $x\in J_t$ where  
$J_t=(\pi_{t-\tau}(\overline m-\varepsilon), \pi_{t-\tau}\overline m)$. 
Using the same argument as in the previous case we check that
$f^*(x)>0$ for $x\in J_t$. Finally,  the inequality $\pi_{t-\tau}\overline m>m$ for sufficiently small $t$ 
implies that $J_t\not\subset A$, which 
contradicts the definition of $A$.     
\end{proof}

\begin{proof}[Proof of Theorem~\ref{Foguel-al}]
Theorem~\ref{Foguel-al} is a simple consequence of
Theorems \ref{asym-th2}, \ref{col-sw} and Lemmas \ref{L:cond-K}, \ref{L:positivity-of-density}.  
\end{proof}

Assumptions (A1)--(A3) are not sufficient to prove asymptotic stability of 
the semigroup  $\{U(t)\}_{t\ge 0}$, 
but according to Theorem~\ref{Foguel-al} we only need to 
check when the semigroup  $\{U(t)\}_{t\ge 0}$ is
\textit{weakly  tight}, i.e. there exists $\kappa>0$ such that
\begin{equation}
\label{k:T}
\sup\limits_{F\in \mathcal F}\limsup_{t\to\infty} \int_F U(t)f(x)\,dx\ge \kappa
\end{equation}
for $f\in D_0$, where $D_0$ is a dense subset of $D$ and $\mathcal F$ is the family of all compact subsets of $X$.
It is clear that weak tightness excludes the case when the semigroup is sweeping from compact sets. 
The process $(\xi_t)_{t\ge 0}$ has the infinitesimal generator  
\begin{equation}
\label{proc-gener}
\mathcal LV(x)=g(x)V'(x)+\Lambda V(Q(x))-\Lambda V(x).  
\end{equation}
The operators $\mathcal A$ and $\mathcal L$ are formally conjugated, i.e. 
\[
\int_0^{\infty} \mathcal Af(x)h(x)\,dx=\int_0^{\infty} f(x)\mathcal Lh(x)\,dx\quad
 \text{for $f\in  \mathfrak D(\mathcal A)$ and $h\in  \mathfrak D(\mathcal L)$}.
\] 
Assume that there exist a $C^1$-function $V\colon [0,\infty)\to [0,\infty)$
and constants  $\varepsilon,\,r,\overline M>0$ such that
\begin{equation}
\label{c:Hf-i2}
 \mathcal LV(x)\le \overline M \quad\textrm{for $\,x< r$} \quad\textrm{and}\quad \mathcal LV(x)\le -\varepsilon \quad\textrm{for $\,x\ge r$}.
\end{equation}
Then the semigroup  $\{U(t)\}_{t\ge 0}$ is weakly tight (see e.g. \cite[page 128]{RT-K-k} for a general result).

Since $V,Q,g$  are $C^1$-functions, the inequality $\mathcal LV(x)\le \overline M$ for $\,x< r$ is obviously fulfilled.
Therefore it remains to check when
 there exists a $C^1$-function $V\colon [0,\infty)\to [0,\infty)$
such that
\begin{equation}
\label{c:Hf}
\limsup_{x\to\infty} [g(x)V'(x)+\Lambda V(Q(x))-\Lambda V(x)]<0.   
\end{equation}
For example, assume that the immune status is 
roughly proportional to the concentration of antibodies  
and their degradation rate is almost constant, then 
$\lim\limits_{x\to\infty} g(x)=-\infty$. Also assume that 
the increase of the concentration of antibodies after the infection
is bounded, i.e. $Q(x)\le x+L$, then condition 
\eqref{c:Hf} is fulfilled with the function 
$V(x)=x$. It means that the semigroup  $\{U(t)\}_{t\ge 0}$ is asymptotically stable. 

Condition \eqref{c:Hf} also holds under much less restrictive assumptions.
For example if $g(x)\le -ax$ and $Q(x)\le bx$ for a sufficiently large $x$,
we can take $V(x)=x^{\gamma}$, $\gamma>0$, and check when 
\[
-a\gamma+\Lambda  b^{\gamma}-\Lambda<0.  
\]
If $a>\Lambda \log b$, then
taking a sufficiently small $\gamma$ we obtain \eqref{c:Hf}.

If $a<\Lambda \log b$, $g(x)\le-ax$ and $Q(x)\ge bx$  
then the semigroup is sweeping from compact sets. Indeed,
consider a negative moment of the process $(\xi_t)_{t\ge 0}$
\[
m_{\gamma}(t)={\rm E\,}\xi_t^{-\gamma}=\int_0^{\infty} x^{-\gamma}u(t,x)\,dx.
\]

One can easy check that
\[
\frac{d}{dt}m_{\gamma}(t)\le c_{\Lambda}m_{\gamma}(t),
\]
where $c_{\gamma}=\gamma a+\Lambda b^{-\gamma}-\Lambda$.
Assume that ${\rm E\,}\xi_0^{-\gamma}<\infty$ (this inequality is fulfilled for example if $\xi_0$ takes values
from some interval $[\alpha,\beta]$, $0<\alpha<\beta<\infty$).   
We have $c_{\gamma}<0$ for a sufficiently small $\gamma$, and consequently   
$\lim_{t\to\infty}m_{\gamma}(t)=0$. But in this case the semigroup
$\{U(t)\}_{t\ge 0}$ is not asymptotically stable and, in consequence,
$\{U(t)\}_{t\ge 0}$ is sweeping from compact sets.

\begin{remark}
\label{rem:tot-var}
Theorem~\ref{Foguel-al} can be formulated in a slightly stronger form. 
Denote by $\nu_t$ the distribution of the process $(\xi_t)_{t\ge 0}$
at time $t$. We do not assume now that the measure $\nu_0$ has a density.
Consider the case when there is an invariant density $f^*$.
Let $\nu^*$ be the measure with density  $f^*$.
Then the measures $\nu_t$ converge to the measure $\nu^*$ in the total variation norm. This result follows from the fact that
if $\nu_t^s$ is the singular part of the measure $\nu_t$, then
$\lim_{t\to\infty}\nu_t^s([0,\infty))=0$.
\end{remark} 

\section{Models with bounded phase spaces}
\label{s:finite}
Now we consider the case when the immune status is a number from the interval $X=[0,M]$.
We start with a version of the model introduced in Section~\ref{s:model}.
We assume that
\begin{itemize}
\item[(B1)] $g\colon [0,M]\to \mathbb R$ is a $C^1$-function such that $g(x)<0$ for $x>0$ and $g(0)=0$,
\item[(B2)] $Q\colon [0,M]\to (0,M]$ is a $C^1$-function such that $Q(x)>x$ for $x\in [0,M)$ and $Q(M)=M$, 
\item[(B3)] $|A|=0 \Longrightarrow  |Q^{-1}(A)|=0$, where $A$ is a Borel subset of $[0,M]$.
\end{itemize}
Then in the same way as in the previous sections we introduce a stochastic semigroup $\{U(t)\}_{t\ge 0}$ on the space $L^1(X,\mathcal B(X),|\cdot|)$
and prove an appropriate version of Theorem~\ref{Foguel-al}. But now $X$ is a compact space and, in consequence, the semigroup $\{U(t)\}_{t\ge 0}$ is not sweeping from compact sets. Therefore we can formulate  the following
\begin{theorem}
\label{As-stab}
The semigroup $\{U(t)\}_{t\ge 0}$ is asymptotically stable.
\end{theorem}

Now we consider a model with an alternative version of the function $Q$ proposed in \cite{DGKT}. We replace assumptions (B2) and (B3) by the following
\begin{itemize}
\item[(B2${}'$)] there exists $K\in (0,M)$ such that   $Q\colon [0,K)\to (0,M]$ is a $C^1$-function such that $x<Q(x)<M$ for $x\in [0,K)$
and $Q(x)=M$ for $x\in [K,M]$, 
\item[(B3${}'$)] $|A|=0 \Longrightarrow  |Q^{-1}(A)|=0$, where $A$ is a Borel subset of $[0,M)$.
\end{itemize}
Illustrative examples of graphs of the transformation $Q$ for both considered cases are given in Fig.~\ref{r:imun-syst2}.     

\begin{figure}
\begin{center}
\begin{picture}(340,120)(20,5)
\includegraphics[viewport=115 554 492 703]{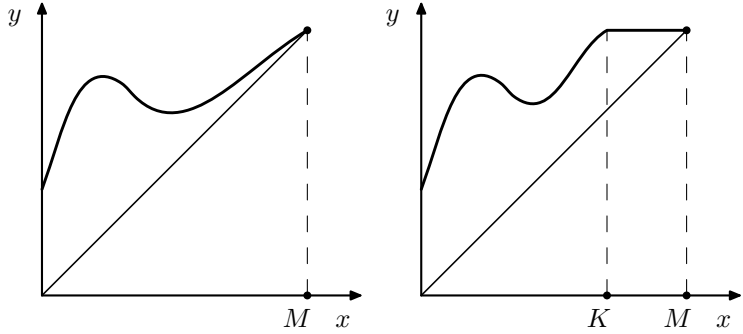}
\end{picture}
\end{center}
\caption{Examples of graphs of $y=Q(x)$.
Left: condition (B2); right: condition (B2${}'$).}
\label{r:imun-syst2}
\end{figure}

Observe that  in this case the transformation $Q$ does not satisfy condition (B3). Indeed, if $A=\{M\}$ then
$|A|=0$ but  the set $Q^{-1}(A)=[K,M]$ has a positive Lebesgue measure. It means that we cannot define the Frobenius-Perron operator $P_Q$
on the space $L^1(X,\mathcal B(X),|\cdot|)$. In order to use introduced earlier apparatus of stochastic semigroups we need to 
modify the definition of the infinitesimal generator $\mathcal A$ of the semigroup $\{U(t)\}_{t\ge 0}$.
The starting point can be the infinitesimal generator $\mathcal L$ of the process $(\xi_t)_{t\ge 0}$
given by \eqref{proc-gener} and we want to find the operator $\mathcal A$ as a formally adjoint operator of $\mathcal L$.
First, we define some modification of the Frobenius-Perron operator.
If  $\widetilde Q=Q\Big|_{[0,K)}$, then according to (B3${}'$) we can define an operator
$P_{\widetilde Q}\colon L^1[0,K]\to L^1[0,M]$ by formula \eqref{def-FP}.
Then $P_{\widetilde Q}$ is a stochastic operator in the sense that it is a linear transformation from 
$L^1[0,K]$ to  $L^1[0,M]$ and $P_{\widetilde Q}$ maps densities to densities.
Next we define the operator $\bar P_{\widetilde Q}\colon L^1[0,M]\to L^1[0,M]$ by $\bar P_{\widetilde Q}f=P_{\widetilde Q}\Big(f\Big|_{[0,K)}\Big)$.
Then $\bar P_{\widetilde Q}$ is a \textit{substochastic operator}, i.e. $\bar P_{\widetilde Q}$ is a positive contraction of $L^1$.
The operator $\mathcal A$ is defined on the set 
\[
\mathfrak D(\mathcal A)=\Big\{f\in L^1[0,M]\colon f'\in L^1[0,M]\quad \text{and \,$g(M)f(M)=-\Lambda \int_K^Mf(x)\,dx$}\Big\} 
\]
and $\mathcal A$ is given by
\[
\mathcal Af= -(gf)'+\Lambda \bar P_{\widetilde Q}f-\Lambda f.
\]
It is not difficult to check that the operators $\mathcal A$ and $\mathcal L$ are formally conjugated.

Now we can write the evolution of densities of the process $(\xi_t)_{t\ge 0}$
in the form of the abstract Cauchy problem \eqref{evol-eq}.
We can treat Eq. \eqref{evol-eq} as abstract notation of a first order
partial differential equation with some linear perturbation  
and some boundary condition. 
Such equations appear in many biological and  physical applications, e.g.
in structured population models~\cite{BPR,M-R,Mac-Tyr,PR-cell-cycle}.   

One can check that the operator $\mathcal A$ generates 
a stochastic semigroup on the space $L^1(X,\mathcal B(X),|\cdot|)$.
The proof of this result is rather standard so we only sketch it omitting the computational part.
Some new and  general results concerning piecewise deterministic Markov processes with boundary can be found in \cite{GMTK}.

We start with some definitions and two general results concerning generators of substochastic and stochastic semigroups.
Let $A$ be a linear operator defined on a linear subspace $\mathfrak D(A)$ of a Banach space $E$.
We say that $\lambda\in\mathbb R$ belongs to the \textit{resolvent set} $\rho(A)$ of $A$,
if the operator $\lambda I-A\colon \mathfrak D(A)\to E$ is invertible. The operator $\mathcal R(\lambda,A):=(\lambda I -A)^{-1}$ for $\lambda\in\rho(A)$
is called the \textit{resolvent operator} of $A$ at $\lambda$.
Now let $E=L^1(X,\Sigma,\mu)$. We call a linear operator $A$  \textit{resolvent positive} if there exists
$\omega\in\mathbb R$ such that $(\omega,\infty)\subseteq\rho(A)$ and
$\mathcal R(\lambda,A)\ge 0$  for all $\lambda>\omega$. Let $L^1_+=\{f\in L^1\colon f\ge 0\}$ and $\mathfrak D(A)_+=\mathfrak D(A)\cap L^1_+$.
A $C_0$-semigroup of substochastic operators on the space $L^1$ is called shortly a \textit{substochastic semigroup}.

\begin{theorem}
\label{t:subst}
A linear operator $A$ with the domain $\mathfrak D(A)\subset L^1$ is the generator of a substochastic semigroup on $L^1$ if and only if
$\mathfrak D(A)$ is dense in $L^1$, the operator  $A$ is resolvent positive, and
\begin{equation}
\label{e:subsem}
\int_X Af(x)\,\mu(dx)\le 0\quad \text{for all }f\in \mathfrak D(A)_+.
\end{equation}
\end{theorem}
The proof of this  result is given e.g. in \cite[Theorem 4.4]{RT-K-k}.  
The second result concerns positive perturbations of substochastic semigroups \cite[Section 6.2]{banasiakarlotti06}.
\begin{theorem}
\label{t:perturb}
Assume that the operator $A_0$ is the generator of a
substochastic semigroup $\{S(t)\}_{t\ge 0}$ on $L^1$ and $B$
is a positive and bounded operator on $L^1$
such that
\begin{equation}
\label{eq:lzero} 
\int_X (A_0f(x)+B f(x))\,\mu(dx)= 0\quad
\text{for}\quad f\in \mathfrak D(A_0)_+.
\end{equation}
Then the operator $A=A_0+ B$ is
the generator of a stochastic semigroup $\{U(t)\}_{t\ge 0}$ on $L^1$.
\end{theorem}

Now we apply Theorems \ref{t:subst} and \ref{t:perturb} to the operator $\mathcal A$.  
Consider the operator  $\mathcal A_0f= -(gf)'-\Lambda f$ with
the domain  $\mathfrak D(\mathcal A_0)=\mathfrak D(\mathcal A)$
and the operator $\mathcal Bf=\Lambda \bar P_{\widetilde Q}f$.
Then $\mathcal B$ is a positive and bounded operator on the space $L^1$.
The operator $\mathcal A_0$ generates a substochastic semigroup $\{S(t)\}_{t\ge 0}$ on the space $L^1$.
This statement is intuitively obvious because the equation 
$u'(t)=\mathcal A_0u(t)$ describes the movement of particles to the left
on the interval $[0,M]$  with the influx of new particles through the right end $M$ with velocity $\int_K^M u(t,x)\,dx$ and the efflux 
from the interval $[0,M]$ with velocity $\int_0^M u(t,x)\,dx$.
The proof that the operator $\mathcal A_0$ generates a substochastic semigroup 
follows from Theorem~\ref{t:subst}. 
It is easy to check that $\mathfrak D(\mathcal A_0)$ is a dense subset of $L^1$ and that condition \eqref{e:subsem} holds.
Then we find that
\[
\mathcal R(\lambda,\mathcal A_0)f(x)=
-\frac{e^{(\lambda+\Lambda)\varphi(x)}}{g(x)}
\left(\Lambda I(\lambda,f)+\int_x^Mf(r)e^{-(\lambda+\Lambda)\varphi(r)}\,dr\right),
\]
where $\varphi(x)=\int\limits_x^M\dfrac{dr}{g(r)}$ and 
$I(\lambda,f)$ is a constant such that 
\[
I(\lambda,f)=\int_K^M \mathcal R(\lambda,\mathcal A_0)f(x)\,dx.
\] 
It is also easy to observe that $\mathcal R(\lambda,\mathcal A_0)\ge 0$  for all $\lambda>0$.
Since 
\[
\begin{aligned}
\int_0^M &(\mathcal A_0f(x)+\mathcal B f(x))\,dx=
\int_0^M \Big(-(gf)'(x)-\Lambda f(x)+\Lambda \bar P_{\widetilde Q}f(x)\Big)\,dx\\
&=-g(M)f(M)-\Lambda\int_0^M f(x)\,dx
+\Lambda\int_0^M  P_{\widetilde Q}\Big(f\Big|_{[0,K)}\Big)(x) \,dx\\
&=\Lambda\int_K^M f(x)\,dx -\Lambda\int_0^M f(x)\,dx+\Lambda\int_0^K f(x)\,dx=0,
\end{aligned}
\]
according to Theorem~\ref{t:perturb} the semigroup $\{U(t)\}_{t\ge 0}$ generated by the operator $\mathcal A$ is a stochastic semigroup.

Theorem~\ref{As-stab} remains true in this case. The only difference in the 
proof is that instead of
formulas \eqref{eq:dpf1b}--\eqref{eq:dpf2b} we need to apply the two following ones
\begin{equation}
\label{eq:dpf1c}
U(t)f=\sum_{n=0}^\infty S_n(t)f, 
\end{equation}
where
\begin{equation}
\label{eq:dpf2c}
S_0(t)f=S(t)f,\quad S_{n+1}(t)f=\int_0^tS_{n}(t-\tau)\mathcal BS_0(\tau)f\,d\tau, \quad
n\ge 0.
\end{equation}




\begin{thebibliography}{99}
\bibitem{banasiakarlotti06} 
\newblock Banasiak J, Arlotti L. 
\newblock \emph{Perturbations of Positive Semigroups with Applications}. 
\newblock London: Springer Monographs in Mathematics. Springer-Verlag; 2006.


\bibitem{BPR} 
\newblock Banasiak J, Pich\'or K, Rudnicki R.
\newblock {Asynchronous exponential growth of a general structured population model}.
\newblock \emph{Acta Appl. Math.} 2012;119:149--166.

\bibitem{davis84} 
\newblock Davis MHA.
\newblock {Piecewise-deterministic Markov processes: A general class of non-diffusion stochastic models}.
\newblock \emph{J. Roy. Statist. Soc. Ser. B} 1984;46:353--388.


\bibitem{DGKT} 
\newblock  Diekmann O, de Graaf WF,  Kretzschmar MEE, Teunis PFM. 
\newblock {Waning and boosting: on the dynamics of immune status}.
\newblock  \emph{J. Math. Biol.} 2018;77:2023--2048.

\bibitem{GKTD} 
\newblock de Graaf WF, Kretzschmar MEE, Teunis PFM, Diekmann O. 
\newblock {A two-phase within-host model for immune response and its application to serological profiles of pertussis}.
\newblock \emph{Epidemics} 2014;9:1--7. 

\bibitem{GMTK} 
\newblock Gwi\.zd\.z P, Tyran-Kami\'nska M. 
\newblock {Densities for piecewise deterministic Markov processes with boundary}.
\newblock \emph{J. Math.\ Anal. Appl.} 2019;479:384--425.


\bibitem{LiM} 
\newblock Lasota A, Mackey MC.
\newblock \emph{Chaos, Fractals and Noise. Stochastic Aspects of Dynamics}.
 \newblock New York: Springer Applied Mathematical Sciences vol. 97 Springer; 1994.

\bibitem{M-R} 
\newblock Mackey MC, Rudnicki R.
\newblock {Global stability in a delayed partial differential equation describing cellular replication}.
\newblock \emph{J.\ Math.\ Biol.} 1994;33:89--109.

\bibitem {Mac-Tyr} 
\newblock Mackey MC,  Tyran-Kami\'nska M.
\newblock {Dynamics and density evolution in piecewise deterministic growth processes}.
\newblock \emph{Ann. Polon. Math.} 2008;94:111--129.


\bibitem {PR-jmaa2} 
\newblock Pich\'or K, Rudnicki R.
\newblock {Continuous Markov semigroups and stability of transport equations}.
\newblock \emph{J. Math.\ Anal.\ Appl.} 2000;249:668--685.

\bibitem{PR-JMMA2016} 
\newblock Pich\'or K, Rudnicki R.
\newblock {Asymptotic decomposition of substochastic operators and semigroups}.
\newblock \emph{J. Math. Anal. Appl.} 2016;436:305--321.

\bibitem{PR-SD2017} 
\newblock Pich\'or K, Rudnicki R.
\newblock {Asymptotic decomposition of substochastic semigroups and applications}.
\newblock \emph{Stoch. Dynam.} 2018;18:1850001, 18 pp.

\bibitem{PR-cell-cycle} 
\newblock Pich\'or K, Rudnicki R.
\newblock {Applications of stochastic semigroups to cell cycle models}.
\newblock \emph{Discrete Contin. Dyn. Syst. Ser. B} 2019;24:2365--2381. 


\bibitem{Rudnicki-LN} 
\newblock  Rudnicki R.
\newblock {Stochastic operators and semigroups and their applications in physics and biology}.
\newblock In: Banasiak J, Mokhtar-Kharroubi M, eds. \emph{Evolutionary Equations with Applications in Natural Sciences}, Lecture Notes in Mathematics, 
vol. 2126: Heidelberg: Springer 2015 (pp. 255--318).

\bibitem{RT-K-k} 
\newblock Rudnicki R, Tyran-Kami\'nska M.
\newblock \emph{Piecewise Deterministic Processes in Biological Models}.
\newblock Cham, Switzerland: SpringerBriefs in Applied Sciences and Technology, Mathematical Methods, Springer;  2017.

\bibitem{TEGB-MK}  
\newblock Teunis PFM, van Eijkeren JCH, de Graaf WF,  Bona\v ci\'c Marinovi\'c A,  Kretzschmar MEE.
\newblock {Linking the seroresponse to infection to within-host heterogeneity in antibody production}.
\newblock \emph{Epidemics} 2016;16:33--39.

\end{thebibliography}




\end{document}